\documentclass[10pt,a4paper]{article}
\usepackage[utf8]{inputenc}
\usepackage{amsmath}
\usepackage{amsfonts}
\usepackage{amssymb}

\usepackage{amsthm}
\usepackage{bbm}

\usepackage{graphicx}

\newtheorem{theorem}{Theorem}[section]
\newtheorem{lemma}[theorem]{Lemma}
\newtheorem{proposition}[theorem]{Proposition}
\newtheorem{corollary}[theorem]{Corollary}
\newtheorem{conjecture}[theorem]{Conjecture}
\theoremstyle{definition}

\newtheorem{definition}[theorem]{Definition}
\newtheorem{remark}[theorem]{Remark}

\newcommand{\R}{{\mathbb{R}}}
\newcommand{\N}{{\mathbb{N}}}
\newcommand{\under}[2]{\smash{\displaystyle#1_{#2}}}

\newcommand{\myMin}[1]{\under{\min}{#1}}
\newcommand{\mySum}[2]{\smash{\displaystyle\sum_{#1}^{#2}}}

\newcommand{\dz}{\dot{z}}

\newcommand{\dg}{\dot{\gamma}}

\newcommand{\ehzcap}{c_{_{\rm EHZ}}}
\newcommand{\Id}{\mathbbm{1}}
\newcommand{\bF}[1]{\mathbf{F}_{#1}}
\newcommand{\kF}{\bF{K}}
\newcommand{\sqba}{c}

\renewenvironment{proof}[1][\proofname]{{\noindent \bfseries #1. }}{\qed}

\begin{document}

\title{On the symplectic size of convex polytopes}
\author{Pazit Haim-Kislev}
\maketitle

\begin{abstract}
In this paper we introduce a combinatorial formula for the Ekeland-Hofer-Zehnder capacity of a convex polytope in $\R^{2n}$. One application of this formula is a certain subadditivity property of this capacity. 
\end{abstract}

\section{Introduction and main results}

Symplectic capacities are well studied invariants in symplectic geometry which, roughly speaking, measure the “symplectic size” of sets (see for example \cite{capacity_survey_1} and \cite{capacity_survey_2}).
The first appearance of a symplectic capacity in the literature (although not under this name) was in Gromov \cite{non_squeezing} where the theory of pseudo-holomorphic curves was developed and implemented.
The concept of a symplectic capacity was later formalized by Ekeland and Hofer in \cite{eh_capacity_def}, where they also gave additional examples using Hamiltonian dynamics. Since then, many other examples of symplectic capacities were constructed.
They are divided, roughly speaking, into ones which are related to obstructions for symplectic embeddings, ones which are defined using pseudo-holomorphic curves, and ones related to the existence and behaviour of periodic orbits in Hamiltonian dynamics.

Two well known examples of symplectic capacities are the Ekeland-Hofer capacity defined in \cite{eh_capacity_def} and the Hofer-Zehnder capacity defined in \cite{hz_capacity}. These two capacities are known to coincide on the class of convex bodies in $\R^{2n}$ (\cite{viterbo_cap}, Proposition 3.10 and \cite{hz_capacity}, Proposition 4). Moreover, in this case they are equal to the minimal action of a closed characteristic on the boundary of the body. In what follows we refer to this quantity as the Ekeland-Hofer-Zehnder capacity (abbreviate by EHZ capacity). See Section \ref{preliminary_section} below, for the definition of a closed characteristic, and for a generalization of this definition to polytopes in $\R^{2n}$. 
We remark that even on the special class of convex bodies in $\R^{2n}$, there are very few methods to explicitly calculate symplectic capacities, specifically the EHZ capacity, and it is in general not an easy problem to find closed characteristics and in particular the minimal ones (cf. \cite{capacity_survey_3}).
The goal of this paper is to give a combinatorial formula for the EHZ capacity for convex polytopes, and discuss some of its applications. 

To state our result we introduce some notations. We work in $\R^{2n}$ with the standard symplectic structure $\omega$. Let $K \subset \R^{2n}$ be a convex polytope with a non-empty interior. Denote the number of $(2n-1)$-dimensional facets of $K$ by $\kF$, and the facets by $\{F_i\}_{i=1}^{\kF}$. Let $h_K(y) := \sup_{x \in K} \langle x,y \rangle$ be the support function of $K$. Denote by $n_i$ the unit outer normal to $F_i$, and $h_i = h_K(n_i)$ the ``oriented height" of $F_i$. Finally, let $S_{\kF}$ be the symmetric group on $\kF$ letters. Our main result is the following. 

\begin{theorem}
\label{formula_theorem}
For every convex polytope $K \subset \R^{2n}$
\begin{equation}
\label{formula_equation}
\ehzcap(K) = \frac{1}{2} \left[ \max_{\sigma \in S_{\kF}, (\beta_i) \in M(K)}  \mySum{1 \leq j < i \leq \kF}{} \beta_{\sigma(i)} \beta_{\sigma(j)} \omega(n_{\sigma(i)},n_{\sigma(j)}) \right]^{-1},
\end{equation}
where
\[ M(K) = \left\{ (\beta_i)_{i=1}^{\kF} : \beta_i \geq 0, \sum_{i=1}^{\kF} \beta_i h_i = 1, \sum_{i=1}^{\kF} \beta_i n_i = 0 \right\}.
\]
\end{theorem}

Note that the maximum runs over $S_{\kF}$, which is a finite set of permutations, and over $M(K)$, which is a finite dimensional convex polytope. Hence the combinatorial nature of the formula. Moreover, this formula allows us (up to computational power) to calculate the capacity of every convex polytope using a computer. We also note that from continuity of the EHZ capacity, some possible applications of Theorem \ref{formula_theorem} about properties of the EHZ capacity on polytopes are automatically extended to all convex bodies (cf. Theorem \ref{subadditivity_theorem} below).

For a centrally symmetric convex polytope $K$ (i.e., when $K=-K$), the above formula can be slightly simplified.
 In this case one can write the normals to the $(2n-1)$-dimensional facets of $K$ as $\{n_1,\ldots,n_{_{\kF'}},-n_1,\ldots,-n_{_{\kF'}}\}$ , where $\kF' = \frac{\kF}{2}$.

\begin{corollary}
\label{centrally_symmetric_formula_theorem}
 For a centrally symmetric convex polytope $K \subset \R^{2n}$,
 
 \[\ehzcap(K) = \frac{1}{4} \left[ \max_{\sigma \in S_{\kF'}, (\beta_i) \in M'(K)} \sum_{1 \leq j < i \leq \kF'} \beta_{\sigma(i)} \beta_{\sigma(j)} \omega(n_{\sigma(i)},n_{\sigma(j)}) \right]^{-1},\]
 where
 \[ M'(K) = \left\{ (\beta_{i})_{i=1}^{\kF'} : \sum_{i=1}^{\kF'} | \beta_i | h_i = \frac{1}{2} \right\} .\]
\end{corollary}

\begin{remark}
\label{reformulationRemark}
We note that Formula \eqref{formula_equation} in Theorem \ref{formula_theorem} can be written as
\begin{equation*}
\ehzcap(K) = \frac{1}{2} \min_{ (\beta_i,n_i)_{i=1}^{\kF} \in M_2(K)  } \left( \sum_{i=1}^{\kF} \beta_i h_K(n_i) \right)^2, 
\end{equation*} 
where
\[ M_2(K) = 
\left\{
	\begin{array}{ll}
		(\beta_i,n_i)_{i=1}^{\kF} : \beta_i \geq 0, (n_i)_{i=1}^{\kF} \text{ are different outer normals to }K\\ \sum_{i=1}^{\kF} \beta_i n_i = 0, \quad\quad
		\sum_{1 \leq j < i \leq \kF} \beta_i \beta_j \omega(n_i,n_j) = 1 \\
	\end{array}
\right\} .\]

In this form of the formula for $\ehzcap(K)$, instead of the permutation $\sigma \in S_{\kF}$ that appeared in \eqref{formula_equation}, we minimize over different orders of the normals, by going over different sequences $(n_i)_{i=1}^{\kF}$. (We refer to Section \ref{more_proofs_section} for the details.)
\end{remark}

\begin{remark}
\label{abbondandolo_remark}
As shown in \cite{Abbondandolo}, using Clarke's dual action principle (see \cite{clarke}), it is possible to express the EHZ capacity of any convex body $K \subset \R^{2n}$ (not necessarily a polytope) as
\[ \ehzcap(K) = \frac{1}{2} \left[ \sup_{z \in \widetilde{\mathcal{E}}} \int_{0}^1  \langle -J \dz, z \rangle \right]^{-1} ,\] 
where
\[ \widetilde{\mathcal{E}} = \left\{ z \in W^{1,2}([0,1],\R^{2n}) : \int_0^1 \dz dt = 0, \dz \in K^{\circ} \right\} ,\]
$K^{\circ} = \{ y \in \R^{2n} : \langle x,y \rangle \leq 1, \text{ for every } x \in K \}$ is the polar body of $K$, and $J$ is the standard complex structure in $\R^{2n}$.
When discretizing this formula, one gets a formula which is similar to the one we get in Theorem \ref{formula_theorem}. However, in this discrete version, as opposed to Theorem \ref{formula_theorem}, one needs to maximize over an infinite dimensional space of piecewise affine loops. The essence of Theorem \ref{formula_theorem}, as will be described later, is that on the boundary of a convex polytope there exists a minimizer with a very specific description, and this enables us to maximize, roughly speaking, over a much smaller space.

\end{remark}

We turn now to describe the main ingredient in the proof of Theorem \ref{formula_theorem}.

Let $K \subset \R^{2n}$ be a convex polytope, and let $\gamma:[0,1] \to \partial K$ be a closed characteristic (for the definition see Section \ref{preliminary_section}). From the definition, if $\gamma(t) \in \text{int}(F_i)$, then $\dg(t)$ must be a positive multiple of $J n_i$ (except maybe for $t$ in a subset of $[0,1]$ of measure zero). Similarly, if $\gamma(t)$ belongs to the intersection of more than one facet, then $\dg(t)$ is a non-negative linear combination of $J n_i$ for $i$ in the participating facets.
A priori, $\gamma(t)$ could return to each facet and each intersection of facets many times. For the purpose of finding the minimal action on the boundary of a convex polytope, we may ignore these options by the following.

\begin{theorem}
\label{simple_loop_theorem}
For every convex polytope $K \subset \R^{2n}$, there exists a closed characteristic $\gamma: [0,1] \to \partial K $ with minimal action such that $\dg$ is piecewise constant and is composed of a finite sequence of vectors, i.e. there exists a sequence of vectors $(w_1,\ldots,w_m)$, and a sequence $(0=\tau_0<\ldots<\tau_{m-1}<\tau_{m}=1)$ so that $\dg(t) = w_i$ for $\tau_{i-1} < t < \tau_{i}$.
Moreover, for each $j \in \{1,\ldots,m\}$ there exists $i \in \{1,\ldots,\kF\}$ so that $w_j = C_j J n_i$ , for some $C_j > 0$, and for each $i \in \{1,\ldots,\kF\}$, the set $\{t : \exists C>0, \dg(t) = C J n_i\}$ is connected, i.e. for every $i$ there is at most one $j \in \{1,\ldots,m\}$ with $w_j = C_j J n_i$.
Hence there are at most $\kF$ points of discontinuity in $\dg$, and $\gamma$ visits the interior of each facet at most once. 
\end{theorem}

Theorem \ref{formula_theorem} follows from the combination of the existence of a simple closed characteristic as described in Theorem \ref{simple_loop_theorem}, and Clarke's dual action principle (see Section \ref{preliminary_section} for the details).

\begin{remark}
There are examples for polytopes with action minimizing closed characteristics which do not satisfy the properties of the closed characteristics one gets from Theorem \ref{simple_loop_theorem}. One example, which can be easily generalized to any convex polytope with an action minimizing closed characteristic passing through a Lagrangian face, is the standard simplex in $\R^4$ where for example on the face $\{ x_1 = 0\} \cap \{ x_2 = 0\}$ one is free to choose a non-trivial convex combination of $e_3$ and $e_4$ as the velocity of an action minimizing closed characteristic, one can also choose it to be equal to $e_3$ for some time, and then to $e_4$, and then $e_3$ again so that the set $\{ t :\exists C>0, \dg(t) = C J n_i\}$ is not connected. 
See \cite{Nir} for a full description of the dynamics of action minimizing closed characteristics on the standard simplex.
\end{remark}

As an application of Theorem \ref{formula_theorem} we solve a special case of the subadditivity conjecture for capacities. This conjecture, raised in \cite{bang_problem}, which is related with a classical problem from convex geometry known as Bang's problem, can be stated as follows:
\begin{conjecture}
If a convex body $K \subset \R^{2n}$ is covered by a finite set of convex bodies $\{K_i\}$ then 
\[ \ehzcap(K) \leq \sum_i \ehzcap(K_i)  .\]
\end{conjecture}
In Section 8 of \cite{bang_problem}, the motivation of this conjecture and its relation with Bang's problem is explained together with some examples. 
It is known that when cutting the euclidean ball $B \subset \R^{2n}$ with some hyperplane into $K_1$ and $K_2$, one has $\ehzcap(B) = \ehzcap(K_1) + \ehzcap(K_2)$. 
The fact that $\ehzcap(B) \geq \ehzcap(K_1) + \ehzcap(K_2)$ was first proved in \cite{ball_add} using an argument involving pseudo-holomorphic curves, and in \cite{bang_problem} it is shown that $\ehzcap(B) \leq \ehzcap(K_1) + \ehzcap(K_2)$.
As a consequence of Theorem \ref{formula_theorem} above, we are able to prove subadditivity for hyperplane cuts of arbitrary convex domains.
\begin{theorem}
\label{subadditivity_theorem}
Let $K \subset \R^{2n}$ be a convex body. Let $n \in S^{2n-1}, c \in \R,$ and $H^- = \{x : \langle x,n\rangle \leq c\}$, $H^+ = \{x: \langle x,n \rangle \geq c\}$. Then for $K_1 = K \cap H^+$ and $K_2 = K \cap H^-$, we have
\[ \ehzcap(K) \leq \ehzcap(K_1) + \ehzcap(K_2).\]
\end{theorem}

The structure of the paper is the following. In Section \ref{preliminary_section} we recall some relevant definitions. In Section \ref{proof_section} we prove Theorem \ref{simple_loop_theorem}, Theorem \ref{formula_theorem} and Corollary \ref{centrally_symmetric_formula_theorem}, and in Section \ref{more_proofs_section} we use Theorem \ref{formula_theorem} to prove Theorem \ref{subadditivity_theorem}.

\textbf{Acknowledgement:}  This paper is a part of the author's thesis, being carried out under the supervision of Professor Shiri Artstein-Avidan and Professor Yaron Ostrover at Tel-Aviv university. I also wish to thank Roman Karasev and Julian Chaidez for helpful comments and remarks.
I am grateful to the anonymous referee for a thorough review and very helpful comments and suggestions.
The work was supported by the European Research Council (ERC) under the European Union Horizon 2020 research and innovation programme [Grant number 637386], and by ISF grant number 667/18.

\section{Preliminaries}
\label{preliminary_section}

\subsection{The EHZ capacity}
Let $\R^{2n}$ be equipped with the standard symplectic structure $\omega$. A normalized symplectic capacity on $\R^{2n}$ is a map $c$ from subsets $U \subset \R^{2n}$ to $[0,\infty]$ with the following properties.
\begin{enumerate}
\item If $U \subseteq V$, $c(U) \leq c(V)$,
\item $c(\phi(U)) = c(U)$ for any symplectomorphism $\phi : \R^{2n} \to \R^{2n}$,
\item $c(\alpha U) = \alpha^2 c(U)$ for $\alpha > 0$,
\item $c(B^{2n}(r)) = c(B^2(r) \times \R^{2n-2}) = \pi r^2$.
\end{enumerate}
For a discussion on symplectic capacities and their properties see e.g. \cite{capacity_survey_1}, \cite{capacity_survey_2} and \cite{hofer_zehnder}.

As mentioned in the introduction, two important examples for symplectic capacities are the Ekeland-Hofer capacity (see \cite{eh_capacity_def}) and the Hofer-Zehnder capacity (see \cite{hz_capacity}). On the class of convex bodies in $\R^{2n}$ (i.e., compact convex sets with non-empty interior), they coincide and we call the resulting function, the EHZ capacity. Moreover, for a smooth convex body, the EHZ capacity equals the minimal action of a closed characteristic on the boundary of the body. Since the focus of this paper is the EHZ capacity, we omit the general definitions of the Hofer-Zehnder and Ekeland-Hofer capacities, and define the EHZ capacity directly.

We start with the definition of a closed characteristic.
Recall that the restriction of the standard symplectic form to the boundary of a smooth domain $\partial \Sigma$, defines a $1$-dimensional subbundle $\text{ker}(\omega | \partial \Sigma)$. A closed characteristic $\gamma$ on $\partial \Sigma$ is an embedded circle in $\partial \Sigma$, whose velocity belongs to $\text{ker}(\omega | \partial \Sigma)$, i.e. $\omega(\dg,v) = 0, \forall v \in T \partial \Sigma$.
This holds if and only if $\dg(t)$ is parallel to $ J n$, where $n$ is the outer normal to $\partial \Sigma$ in the point $\gamma(t)$, and $J$ is the standard complex structure.

From the dynamical point of view, a closed characteristic is any reparametrization of a periodic solution to the Hamiltonian equation $\dg(t) = J \nabla H (\gamma(t))$, for a smooth Hamiltonian function $H : \R^{2n} \to \R$ with $H |_{\partial \Sigma} = c$, and $H |_{\Sigma} \leq c$ for some $c \in \R$ a regular value of $H$. We call these periodic solutions closed Hamiltonian trajectories.

We recall that the action of a closed loop $\gamma : [0,T] \to \R^{2n}$ is defined by 
\[A(\gamma) := \frac{1}{2} \int_0^T \langle J \gamma(t), \dg(t) \rangle dt,\]
 and it equals the symplectic area of a disc enclosed by $\gamma$.
 
The EHZ capacity of a smooth convex body $K \subset \R^{2n}$ is
\[\ehzcap(K) = \min \{ A(\gamma) : \gamma \text{ is a closed characteristic on } \partial K \}.\]
It is known that the minimum is always attained (see \cite{eh_capacity_def}, \cite{hofer_zehnder}). 
One can extend this definition by continuity to non-smooth convex domains with non-empty interior. We elaborate in the next section (see e.g. \cite{singular_survey_1}, \cite{singular_survey_2}).

\subsection{The case of convex polytopes}
\label{preliminary_polytopes_section}
The explicit definition of the EHZ capacity above was given only for smooth bodies, and extended by continuity to all convex domains with non-empty interior. It turns out that also in the case of a non-smooth body, the capacity is given by the minimal action of a closed characteristic on the boundary of $K$ (see \cite{singular_capacity}), however one then needs to discuss generalized closed characteristics. Let us state this precisely here for the case of convex polytopes.

Let $K \subset \R^{2n}$ be a convex polytope. For the following discussion suppose that the origin $0$ belongs to $K$.
Recall that we denote the $(2n-1)$-dimensional facets of $K$ by $\{F_i\}_{i=1}^{\kF}$, and their outward unit normals by $\{n_i\}_{i=1}^{\kF}$.
Let $x \in \partial K$.
We define the outward normal cone of $K$ at $x$ to be $N_K(x) := \R_+ \text{conv}\{n_i: x \in F_i \}$ (for the definition of the outward normal cone for a general convex body see \cite{singular_survey_2}).
Recall that $W^{1,2}([0,1],\R^{2n})$ is the Hilbert space of absolutely continuous functions whose derivatives are square integrable. We equip this space with the natural Sobolev norm:
\[ \| z\|_{1,2} := \left( \int_0^1 \| z(t) \|^2 + \| \dz(t) \|^2 dt \right)^{\frac{1}{2}} .\]
\begin{definition}
\label{ccDef}
A \begin{it}closed characteristic\end{it} on $\partial K$ is a closed loop $\gamma \in W^{1,2}([0,1],\R^{2n})$ which satisfies $\text{Im}(\gamma) \subset \partial K$, and $\dg(t) \in J N_K(\gamma(t))$ for almost every $t \in [0,1]$.
\end{definition}

We remark that the condition $\text{Im}(\gamma) \subset \partial K$ can be weakened to $\gamma(0) \in \partial K$, since the assumption on $\dg$ and the fact that $\gamma$ is a closed loop already imply that $\gamma(t) \in \partial K$ for each $t$ (see \cite{singular_survey_2}).

Definition \ref{ccDef} also has a Hamiltonian dynamics interpretation. Let $H$ be a Hamiltonian function for which $K$ is a sub-level set, and $\partial K$ is a level set. Just like in the smooth case, (generalized) closed Hamiltonian trajectories of the Hamiltonian $H$ on $\partial K$, are reparametrizations of closed characteristics on $\partial K$, and upto a reparametrization, every closed characteristic is a closed Hamiltonian trajectory, only instead of $\dg(t) = J \nabla H(\gamma(t))$, the Hamiltonian equation becomes an inclusion
\[\dg(t) \in J \partial H(\gamma(t)) \text{ almost everywhere},\]
where $\partial H$ is the subdifferential of $H$ (see e.g. \cite{subdiff_1}). We remark that if $H$ is smooth at the point $x$, then $\partial H(x) = \{ \nabla H(x) \}$, and hence if $H$ is smooth the two Hamiltonian equations coincide.
For simplicity, we shall work with a specific Hamiltonian function. Denote the gauge function of $K$ by 
\[ g_K(x) = \inf\{\lambda : \frac{x}{\lambda} \in K\},\]
and consider the Hamiltonian function $g_K^2$. 
Note that $g_K^2 | _{\partial K} = 1$. For each $1 \leq i \leq \kF$ let $p_i = J \nabla (g_K^2)(x)$, for a point $x \in \text{int}(F_i)$. 
It is easily seen that the subdifferential of $g_K^2$ at the point $x \in \partial K$ is equal to
\[  \text{conv}\{ \nabla (g_K^2)|_{\text{int}(F_i)} : x \in F_i \},\]
which implies
\[ J \partial g_K^2(x) = \text{conv}\{ p_i : x \in F_i \}.\]

To conclude, for a convex polytope $K \subset \R^{2n}$, the EHZ capacity is the minimal action over all periodic solutions $\gamma \in W^{1,2}([0,T],\partial K)$, to the Hamiltonian inclusion: 
\[ \dg(t) \in \text{conv}\{ p_i : \gamma(t) \in F_i \} \text{ almost everywhere}.\]
\subsection{Clarke's dual action principle}

Let $K \subset \R^{2n}$ be a convex body (not necessarily smooth). Recall that the support function of $K$ is $h_K(x) = \sup \{ \langle y, x \rangle ; y \in K \}$. Note that $h_K$ is the gauge function of $K^\circ$ and that $4^{-1}g_K^2$ is the Legendre transform of $h_K^2$ (see e.g. \cite{singular_capacity}).

Following Clarke (see \cite{clarke}), we look for a dual variational principle where solutions would correspond to closed characteristics (cf. \cite[Section 1.5]{hofer_zehnder}). 
Consider the problem 
\[ \min_{z \in \mathcal{E}} \int_0^1 h_K^2(-J \dz(t)) dt,\]
where
\[ \mathcal{E} = \left\{ z \in W^{1,2}([0,1],\R^{2n}) : \int_0^1 \dz(t) dt = 0, \int_0^1 \langle -J\dz(t),z(t) \rangle dt = 1 \right\} .\]
Define
\[ I_K(z) = \frac{1}{4} \int_0^1 h_K^2(-J \dz(t)) dt .\]
Let 
\[ \mathcal{E}^\dagger = \left\{ z \in \mathcal{E} : \exists \alpha \in \R^{2n} \text{ such that } 8 I_K(z) z + \alpha \in \partial h_K^2(-J \dz) \right\}.\]
This is the set of weak critical points of the functional $I_K$ (see \cite{singular_capacity}).
The following lemma is an adjustment of the dual action principle to the non-smooth case, and it appears e.g., as Lemma 5.1 in \cite{singular_capacity}.
\begin{lemma}
\label{dual_bijection_lemma}
Let $K \subset \R^{2n}$ be a convex polytope.
There is a correspondence between the set of 
closed characteristics $\gamma$ on $\partial K$, and the set of elements $z \in \mathcal{E}^{\dagger}$. Under this correspondence, there exist $\lambda \in \R^+$, and $b \in \R^{2n}$ so that $z = \lambda \gamma + b$ and moreover $A(\gamma) = 2I_K(z)$. In particular, any minimizer $z \in \mathcal{E}$ of $I_K(z)$ belongs to $\mathcal{E}^{\dagger}$ and therefore has a corresponding closed characteristic with minimal action.
\end{lemma}

\section{Action minimizing orbits on polytopes}
\label{proof_section}

We start with the proof of Theorem \ref{simple_loop_theorem}. Let us first describe the idea of the proof. 
We start from a closed characteristic with minimal action, and consider its corresponding element $z \in \mathcal{E}^\dagger$ (see Lemma \ref{dual_bijection_lemma}). We then approximate it with a certain sequence of piecewise affine loops. By piecewise affine we mean that the velocity of the loop $z$ can be written as $\dz(t) = \sum_{j=1}^m \Id_{I_j}(t) w_j$ for almost every $t \in [0,1]$, where $(I_j)_{j=1}^m$ is a partition of $[0,1]$ into intervals (see Definition \ref{partition_definition} below) and $(w_j)_{j=1}^m$ is a finite sequence of vectors which we call the velocities of $z$.
Our goal is to construct from each piecewise affine loop in the approximating sequence a new simple loop in the sense of the requirements of Theorem \ref{simple_loop_theorem}, i.e. that the sequence $(w_j)_{j=1}^m$ is composed of positive multiples of $J n_i$, where $n_i$ is some outer normal vector to a $2n-1$-dimensional facet of $K$, and that for each $i = 1,\ldots \kF$ there is at most one $j$ so that $w_j$ is a positive multiple of $J n_i$. The limit of these simple loops gives us the desired minimizer of $I_K$ and by invoking Lemma \ref{dual_bijection_lemma} again we get the desired closed characteristic.
In order to construct a simple loop from each piecewise affine loop $z$, we make two changes to it which are described in Lemma \ref{no_linear_combination_lemma} and Lemma \ref{one_speed_lemma} below. 
Recall that the velocities $w_j$ are positive linear combinations of $J n_i$, $i=1,\ldots,\kF$, and additionally, maybe after a reparametrization, one can write $w_j = c \sum_{i=1}^l a_{j_i} p_{j_i}$, $\sum_{i=1}^l a_{j_i} = 1$, where $c > 0$ is a constant independent of $i$, and $p_i,i=1,\ldots,\kF$ are the vectors described in Section \ref{preliminary_polytopes_section} above.
The first change, roughly speaking, takes a time segment $I$ of the loop $z$ where the velocity is a convex combination of $\{c \cdot p_{j_i}\}_{i=1}^l$ and changes it to a sequence of $l$ segments where in each segment the velocity is $c \cdot p_{j_i}$, and the time of each segment is $a_{j_i} |I|$ (see Figure \ref{fig1}).
\begin{figure}
	\includegraphics[scale=0.6]{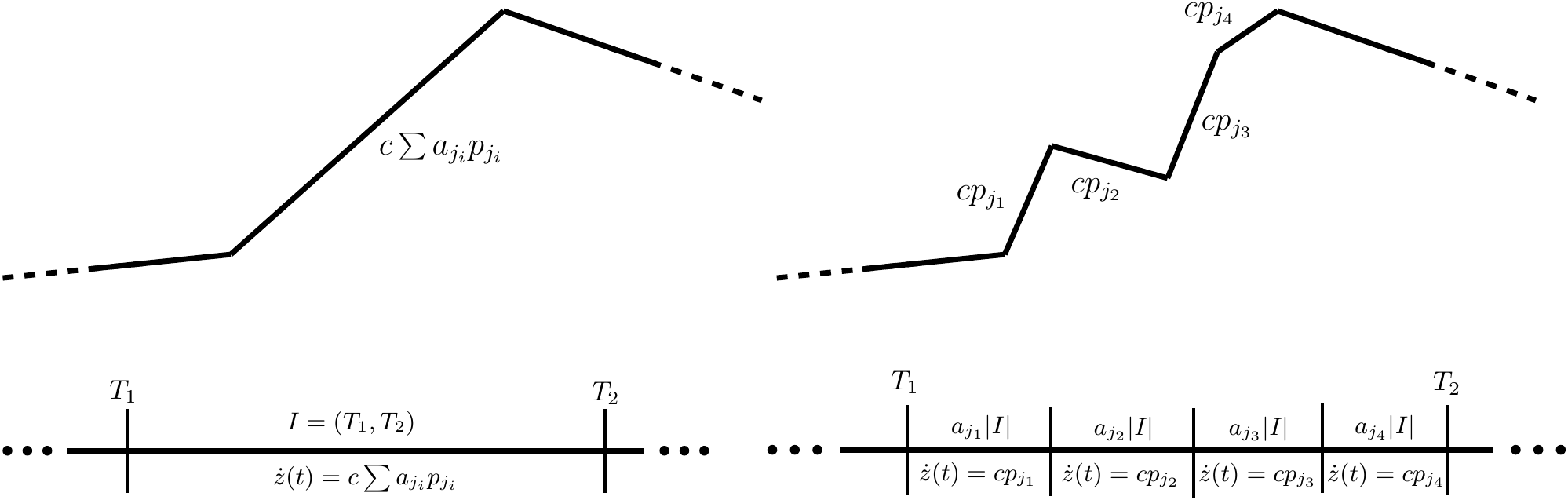}
	\caption{Description of the first change to the loop $z	$: We break a convex combination and move in each velocity separately}
\label{fig1}
\end{figure}
In addition, we show that one can choose the order of $\{ p_{j_i} \}_{i=1}^l$ to make sure that the value of $\int_0^1 \langle -J \dz,z \rangle dt$ does not decrease.
The second change changes the order of the velocities and, roughly speaking, moves all the time segments where the velocities are proportional to a certain $J n_i$ to become adjacent to one another (see Figure \ref{fig2}).
\begin{figure}
\includegraphics[scale=0.6]{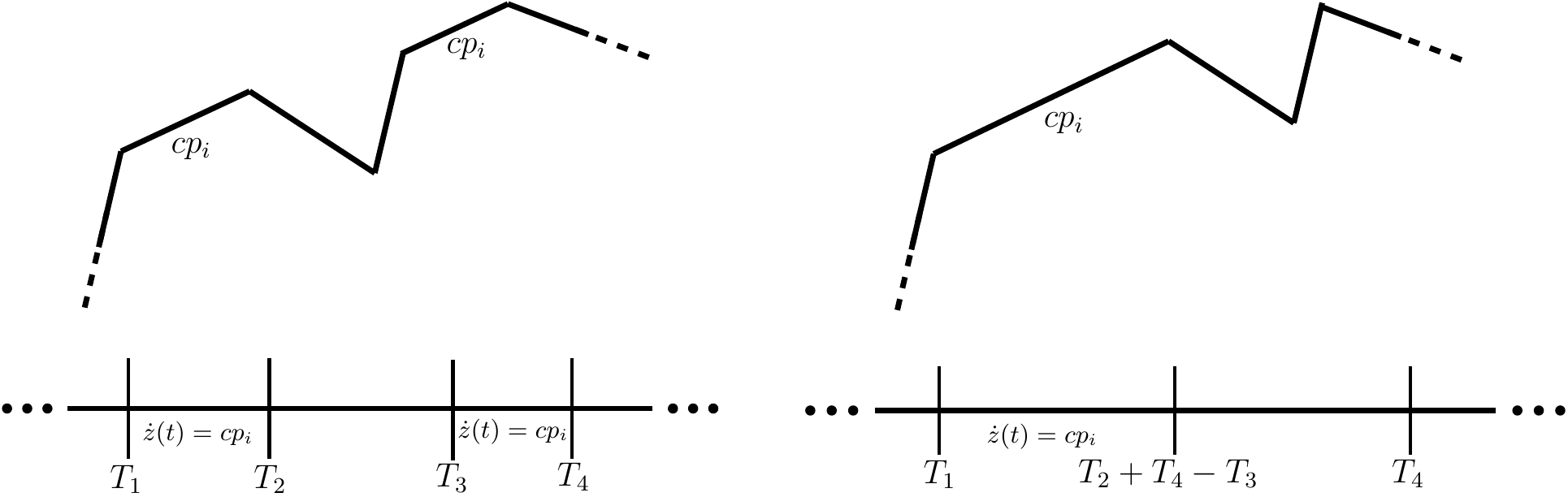}
\caption{Description of the second change to the loop $z$: We bring segments of the loop where it moves in the same velocity together}
\label{fig2}
\end{figure}
This change thus ensures that the set $\{ t : \dz(t) \text{ is a positive multiple of } J n_i \}$ is connected for every $i =1,\ldots,\kF$, i.e. that for each $i$, there is at most one $j$ so that $w_j$ is a positive multiple of $J n_i$. In addition, one can do this change while ensuring that the value of $\int_0^1 \langle -J \dz,z \rangle dt$ does not decrease.
Finally, after dividing the simple loop by $\int_0^1 \langle -J \dz,z \rangle dt$, one gets an element in $\mathcal{E}$ whose value under $I_K$ does not increase, and hence it is still a minimizer. This loop, by virtue of Lemma \ref{dual_bijection_lemma}, gives the required simple closed characteristic.

We begin by describing the piecewise affine approximation.

\begin{lemma}
\label{lemma_affine_approximation}
Fix a set of vectors $v_1,\ldots,v_k \in \R^{2n}$.
Suppose $z \in W^{1,2}([0,1],\R^{2n})$ satisfies that for almost every $t \in [0,1]$,
one has $\dz(t) \in \text{conv}\{v_1,\ldots,v_k\}$.
Then for every $\varepsilon > 0$, there exists a piecewise affine function $\zeta$ with $\|z - \zeta \|_{1,2} < \varepsilon$, and so that $\dot{\zeta}$ is composed of vectors from the set $\text{conv}\{v_1,\ldots,v_k\}$, and $\zeta(0) = z(0), \zeta(1) = z(1)$.
\end{lemma}
\begin{proof}
Let $\varepsilon > 0$. 
Using \cite{schaftingen}, there exists a partition $0=t_1<t_2<\ldots<t_m = 1$ of $[0,1]$ so that the piecewise affine function $\zeta$ defined by the requirements that for each $i=1,\ldots,m-1$, the restriction $\zeta |_{(t_i,t_{i+1})}$ is affine, and $\zeta(t_i) = z(t_i)$, satisfies that $\|z - \zeta \|_{1,2} < \varepsilon$.
We are left with showing that $\dot{\zeta}(t) \in \text{conv}\{v_1,\ldots,v_k\}$.
Note that for $t \in (t_i,t_{i+1})$,
\[ \dot{\zeta}(t) = \frac{z(t_{i+1}) - z(t_i)}{t_{i+1} - t_i} = \frac{\int_{t_i}^{t_{i+1}} \dz(t) dt}{t_{i+1} - t_i} .\]
It is a standard fact that there exists a sequence $\left(\xi_j^{(N)}\right)_{N \in \N, j \in \{1,\ldots,N\}}$ so that 
$\frac{\xi_j^{(N)} - t_i}{t_{i+1}-t_{i}} \in [\frac{j-1}{N},\frac{j}{N}]$, and
\[ \frac{\int_{t_i}^{t_{i+1}} \dz(t) dt}{t_{i+1} - t_i} = \lim_{N \to \infty} \sum_{j=1}^N \frac{\dz(\xi_j^{(N)})}{N} .\]
Note that for each $N$, one has $\sum_{j=1}^N \frac{\dz(\xi_j^{(N)})}{N} \in \text{conv}(\text{Im}(\dz)) \subset \text{conv}\{v_1,\ldots,v_k\}$. This observation together with the fact that $\text{conv}\{v_1,\ldots,v_k\}$ is closed, gives $\dot{\zeta}(t) \in \text{conv}\{v_1,\ldots,v_k\}$. 

\end{proof}

\begin{definition}
\label{partition_definition}
We call a finite sequence of disjoint open intervals $(I_i)_{i=1}^m$ a partition of $[0,1]$, if there exists an increasing sequence of numbers $0 = \tau_0 \leq \tau_1 \leq \ldots \leq \tau_m = 1,$ with $I_i = (\tau_{i-1},\tau_i)$. 
\end{definition}

The following proposition will be helpful later.
\begin{proposition}
\label{technicalProp}
Let $z \in W^{1,2}([0,1],\R^{2n})$ be a closed loop such that $\dz(t) = \sum_{i=1}^m \Id_{I_i}(t) w_i$ almost everywhere, where $\left(I_i = (\tau_{i-1},\tau_i) \right)_{i=1}^m$  is a partition of $[0,1]$, and $w_1,\ldots,w_m \in \R^{2n}$.  Then
\[ \int_0^1 \langle -J \dz, z \rangle dt = \sum_{i=1}^m \sum_{j = 1}^{i-1} |I_j| |I_i| \omega(w_i,w_j). \]
\end{proposition}

\begin{proof}
\begin{align*}
\int_0^1 \langle -J \dz, z \rangle dt  &= \int_0^1 \langle -J \dz , z(0) + \int_0^t \dz(s) ds \rangle dt \\
 &= \int_0^1 \langle -J \sum_{l=1}^m \Id_{I_l}(t) w_l , \int_0^t \sum_{l=1}^m \Id_{I_l}(s) w_l ds \rangle dt \\
 &= \sum_{i=1}^m \int_{I_i} \langle -J \sum_{l=1}^m \Id_{I_l}(t) w_l , \int_0^{\tau_{i-1}} \sum_{l=1}^m \Id_{I_l}(s) w_l ds + \int_{\tau_{i-1}}^t w_i ds \rangle dt \\
 &= \sum_{i=1}^m \int_{I_i} \langle -J w_i , \sum_{j=1}^{i-1} \int_{I_j} \sum_{l=1}^m \Id_{I_l}(s)w_l ds + (t-\tau_{i-1})w_i \rangle dt \\
 &= \sum_{i=1}^m \int_{I_i} \langle -J w_i , \sum_{j=1}^{i-1} \int_{I_j} w_j ds \rangle dt = \sum_{i=1}^m \sum_{j=1}^{i-1} |I_i| |I_j| \omega(w_i,w_j).
\end{align*}
\end{proof}

\begin{lemma}
\label{no_linear_combination_lemma}
Fix a set of vectors $ v_1, \ldots, v_k \in \R^{2n} $. Let $z \in W^{1,2}([0,1],\R^{2n})$ be a piecewise affine loop, where $\dz(t) \in \text{conv}\{v_1,\ldots,v_{{k}}\}$ for almost every $t \in [0,1]$, then there exists another piecewise affine loop $z' \in W^{1,2}([0,1],\R^{2n})$ so that $\dz'(t) \in \{ v_1,\ldots,v_{{k}}\}$ for almost every $t$, and 
\[ \int_0^1 \langle -J \dz',z' \rangle dt \geq \int_0^1 \langle -J \dz, z \rangle dt .\]
\end{lemma}

\begin{proof}
The idea of the proof is to replace any convex combination of $\{v_i\}_{i=1}^k$ in the velocity of $z$ by moving in each velocity $v_i$ separately.
Write $\dz(t) = \sum_{j=1}^m \Id_{I_j}(t) w_j$, where for each $j$, $w_j \in \text{conv}\{v_1,\ldots,v_{{k}}\}$, and $\left(I_j\right)_{j=1}^m$ is a partition of $[0,1]$. Suppose that $w_i = \sum_{j=1}^{l} a_{i_j} v_{i_j}$, where $a_{i_j} > 0$, $i_j \in \{1,\ldots,k\}$, and $l \in \mathbb{N}$ dependent on $i$.
Note that $\sum_{j=1}^{l} a_{i_j} = 1$. 
Consider the partition of $I_i$ to disjoint subintervals $I_{i_j} \subset I_i$ for every $j=1,\ldots,l$ where the length of $I_{i_j}$ is $|I_{i_j}| = a_{i_j} |I_i|$.
Define the following loop
\begin{equation}
\label{z_velocites_equation}
\dz'(t) = \sum_{j=1}^{i-1} \Id_{I_j}(t) w_j + \sum_{j=1}^l \Id_{I_{i_j}}(t) v_{i_j} + \sum_{j=i+1}^m \Id_{I_j}(t) w_j . 
\end{equation}
We shall specify the order of the subintervals $I_{i_j}$'s and the velocities $v_{i_j}$'s appearing in \eqref{z_velocites_equation} later. 
It follows immediately that $\int_0^1 \dz'(t) dt = \int_0^1 \dz(t) dt = 0$. 
Next we show that, if the order of the vectors $v_{i_j}$ is properly chosen, then
\begin{equation*}
\int_0^1 \langle -J \dz' , z' \rangle dt \geq \int_0^1 \langle -J \dz, z \rangle dt .
\end{equation*}
Indeed, by Proposition \ref{technicalProp},
\begin{align*}
\int_0^1 \langle -J \dz' , z' \rangle dt & = \sum_{\substack{r < s \\ r,s \neq i}} |I_r| |I_s| \omega(w_s,w_r) + \sum_{j=1}^l \sum_{r<i} |I_r| |I_i| a_{i_j} \omega(v_{i_j}, w_r)  &\\
 & + \sum_{j=1}^l \sum_{r > i} |I_r| |I_i| a_{i_j} \omega(w_r, v_{i_j}) +  \sum_{1 \leq r < s \leq l} |I_i|^2 a_{i_r} a_{i_s} \omega(v_{i_s},v_{i_r}) &\\
& = \sum_{\substack{r < s \\ r,s \neq i}} |I_r| |I_s| \omega(w_s,w_r) + \sum_{r<i} |I_r| |I_i| \omega(w_i, w_r) &\\
& + \sum_{r>i} |I_r| |I_i| \omega(w_r, w_i) +  \sum_{1 \leq r < s \leq l} |I_i|^2 a_{i_r} a_{i_s} \omega(v_{i_s},v_{i_r}) &\\
& = \int_0^1 \langle -J \dz , z \rangle dt + \sum_{1 \leq r < s \leq l} |I_i|^2 a_{i_r} a_{i_s} \omega(v_{i_s},v_{i_r}). &\\ 
\end{align*}
Finally, we wish to prove that 
\begin{equation}
\label{positive_sum}
\sum_{1 \leq r < s \leq l} a_{i_r} a_{i_s} \omega(v_{i_s},v_{i_r}) \geq 0. 
\end{equation}
Note that we are free to select the order of $v_{i_1}, v_{i_2}, \ldots, v_{i_l}$. If we reverse the order of the velocities we get that the sum in \eqref{positive_sum} changes sign. Therefore, by rearranging the $v_{i_j}$'s in \eqref{z_velocites_equation} one can choose the order so that inequality \eqref{positive_sum} would hold.
By applying this argument to all intervals $I_i$ one gets the thesis.
\end{proof}

\begin{lemma}
\label{one_speed_lemma}
Fix a finite sequence of pairwise distinct vectors $(v_1,\ldots,v_k).$ Let $z \in W^{1,2}([0,1],\R^{2n})$ be a piecewise affine loop so that $\dz(t) = \sum_{i=1}^m \Id_{I_i}(t) w_i$, where $\left( I_i = (\tau_{i-1},\tau_i) \right)_{i=1}^m$ is a partition of $[0,1]$, and for each $i$, $w_i \in \{v_1,\ldots,v_k\}$. Then there exists another piecewise affine loop $z'$ so that $\dz'(t) \in \{v_1,\ldots,v_k\}$ for almost every $t$, and $\{t : \dz'(t) = v_j\}$ is connected for every $j = 1,\ldots,k$. In addition,
\[ \int_0^1 \langle -J \dz',z' \rangle dt \geq \int_0^1 \langle -J \dz, z \rangle dt .\] 

\end{lemma}

\begin{proof}
Assume that for some $r<s$ one has $w_r = w_s$, consider a rearrangement of the intervals $I_i$ where we erase the interval $I_s$ and increase the length of the interval $I_r$ by $|I_s| = \tau_s - \tau_{s-1}$, more precisely,
\[ I_i' = 
\left\{ \begin{matrix} 
(\tau_{i-1},\tau_i), & & i < r \\
(\tau_{i-1},\tau_i + \tau_s - \tau_{s-1}), & & i = r \\
(\tau_{i-1} + \tau_s - \tau_{s-1}, \tau_i + \tau_s - \tau_{s-1}), & & r < i < s \\
\emptyset, & & i = s \\
(\tau_{i-1},\tau_i), & & i > s
\end{matrix} \right. 
\]
Now define $z'$ by $\dz'(t) = \sum_{i=1}^m \Id_{I_i'}(t) w_i$. 
We will show that the action of this loop $z'$ or the analogous loop $z''$ which is defined by erasing $I_r$ and increasing the length of $I_s$ by $|I_r|$ is not smaller than the action of $z$.
First note that 
\[ 0 = \int_0^1 \dz dt = \sum_{i=1}^m |I_i| w_i , \]
while
\[ \int_0^1 \dz' dt = \sum_{i=1}^m |I_i'| w_i .\]
Since $w_r = w_s$ the two sums are only different in the order of summation and thus equal.
Next, we claim that
\begin{equation}
\label{action_inequality}
\int_0^1 \langle -J \dz', z' \rangle dt \geq \int_0^1 \langle -J \dz, z \rangle dt.
\end{equation}
By Proposition \ref{technicalProp},
\begin{equation}
\label{action_z_equation}
\int_0^1 \langle -J \dz, z \rangle dt = \sum_{i=1}^m \sum_{j < i} |I_j| |I_i| \omega(w_i,w_j).
\end{equation}
Consider the change in $\int_0^1 \langle -J \dz, z \rangle dt$ after removing $I_s$ and adding $|I_s|$ to the length of $I_r$. Since $w_r = w_s$, the coefficient of $\omega(w_r,w_i)$ does not change for $i < r$ or $i > s$. For $r < i < s$ instead of the term $|I_s| |I_i| \omega(w_s,w_i)$ in \eqref{action_z_equation} we add $|I_s| |I_i| \omega(w_i, w_s)$ to the term $|I_r||I_i| \omega(w_i,w_r)$, so the action difference is 
\[\int_0^1 \langle -J \dz', z' \rangle dt - \int_0^1 \langle -J \dz, z \rangle dt  = \sum_{i = r+1}^{s-1} 2 |I_s| |I_i| \omega(w_i,w_s) .\]
Note that if one erases $I_r$ and increases the length of $I_s$ by $|I_r|$ instead, the action difference becomes
\[ \sum_{i = r+1}^{s-1} 2 |I_r| |I_i| \omega(w_r,w_i) , \]
 which has an opposite sign, and hence either $z'$ or $z''$ satisfies \eqref{action_inequality}.
Finally, we continue to join different disjoint intervals $I_r$,$I_s$ whenever $w_r = w_s=v_i$ by induction, until $\{t : \dz'(t) = v_i \}$ is connected for every $i=1,\ldots,k$.
\end{proof}

\begin{proposition}
\label{IKprop}
Let $K \subset \R^{2n}$ be a convex polytope so that the origin $0$ belongs to $K$. Let $\{n_i\}_{i=1}^{\kF}$ be the normal vectors to the $2n-1$-dimensional facets of $K$, and let $p_i = J \partial g_K^2 |_{F_i} = \frac{2}{h_i} J n_i$.
Recall that $h_i := h_K(n_i)$.
Let $c > 0$ be a constant and let $z \in \mathcal{E}$ be a loop that satisfies that for almost every $t$, there is a non-empty face of $K$, $F_{j_1} \cap \ldots \cap F_{j_l} \neq \emptyset$,  with $\dz(t) \in c \cdot \text{conv}\{p_{j_1},\ldots,p_{j_l}\}$.
Then
\[ I_K(z) = c^2 .\]
\end{proposition}
\begin{proof}
Fix $t_0 \in [0,1]$ and assume that $\dz(t_0) = c \cdot \sum_{i=1}^l a_i p_{j_i}$ for $a_i \geq 0$, $\sum_{i=1}^l a_i = 1$.
By the definition of $h_K$ one has
\[ h_K(-J \dz(t_0)) = h_K(2c \sum_{i=1}^l \frac{a_i}{h_{j_i}} n_{j_i} ) = \sup_{x \in K} \langle x , 2c \sum_{i=1}^l \frac{a_i}{h_{j_i}} n_{j_i} \rangle = 2c \sup_{x \in K} \sum_{i=1}^l \frac{a_i}{h_{j_i}} \langle x, n_{j_i} \rangle. \]
On the other hand $ \sup_{x \in K} \langle x , n_{j_i} \rangle = h_{j_i} $, and it is attained for every $x \in F_{j_i}$.
Hence for any choice of $y \in F_{j_1} \cap \ldots \cap F_{j_l}$, 
\[ \sup_{x \in K} \sum_{i=1}^l \frac{a_i}{h_{j_i}} \langle x , n_{j_i} \rangle = \sum_{i=1}^l \frac{a_i}{h_{j_i}}  \langle y , n_{j_i} \rangle = \sum_{i=1}^l a_i = 1. \]
Hence
\[ h_K(-J \dz(t)) = 2c,\]
for almost every $t$, and
\[ I_K(z) = \frac{1}{4} \int_0^1 h_K^2(-J\dz(t))dt = c^2.\]
\end{proof}

\begin{proof}[Proof of Theorem \ref{simple_loop_theorem}]
Since the existence of a closed characteristic with the desired properties is independent on translations, we assume without loss of generality that the origin $0$ belongs to $K$ (see also Remark \ref{translation_invariant_remark}).
Assume that $\gamma: [0,1] \to \partial K$ is a closed characteristic with minimal action such that
$ \dg(t) \in dJ \partial g^2_K(\gamma(t))$ for almost every $t$, where $d > 0$ is a constant independent of $t$ 
(recall that every closed characteristic equals upto a reparametrization to a solution to the Hamiltonian inclusion $ \dg(t) \in J \partial g^2_K(\gamma(t)) $ almost everywhere, and one can reparametrize by some constant $d$ to get $\gamma(0) = \gamma(1)$).
From Lemma \ref{dual_bijection_lemma} it follows that there is $z \in \mathcal{E}^\dagger$ such that $A(\gamma) = 2I_K(z)$, and $z = \lambda \gamma + b$, with some constants $\lambda \in \R^+, b \in \R^{2n}$. 
Note that $\dz(t) = \lambda \dg(t) \in \lambda d \cdot \text{conv}\{p_1,\ldots,p_{_{\kF}}\}$, and denote $c = \lambda d$. Moreover, $z$ satisfies the conditions of Proposition \ref{IKprop} and hence $I_K(z) = c^2$.
From Lemma \ref{lemma_affine_approximation} for every $N \in \mathbb{N}$ one can find a piecewise affine loop $\zeta_N$ such that $\| z - \zeta_N \|_{1,2} \leq \frac{1}{N}$ and $\dot{\zeta}_N(t) \in c \cdot \text{conv}\{p_1,\ldots,p_{_{\kF}}\}$ for almost every $t$. 
By applying first Lemma \ref{no_linear_combination_lemma} with $v_i=c p_i,i=1,\ldots,\kF$ to $\zeta_N$, and then take the result and apply to it Lemma \ref{one_speed_lemma} again with $v_i=c p_i,i=1,\ldots,\kF$, one gets a piecewise affine loop $z_N$ which can be written as 
\[ \dz_N(t) = \sum_{i=1}^{m_{_N}} \Id_{I^N_i} (t) v^N_i ,\]
where $v^N_i = c \cdot p_j$ for some $j \in \{1,\ldots,\kF\}$ and for every $j$ there is at most one such $i$.
Moreover one has
\[ A_N := \sqrt{ \int_0^1 \langle -J \dz_N, z_N \rangle dt } \geq \sqrt{ \int_0^1 \langle -J \dot{\zeta}_N ,\zeta_N \rangle dt }.\]
Hence denote $z_N' = \frac{z_N}{A_N} \in \mathcal{E}$, and write $w^N_i = \frac{v^N_i}{A_N}$ for the velocities of $z_N'$, and write $c_N = \frac{c}{A_N}$.
The fact that $\zeta_N \xrightarrow{N \to \infty} z$ implies that $\int_0^1 \langle -J \dot{\zeta}_N ,\zeta_N \rangle dt \xrightarrow{N \to \infty} 1$. Hence $\lim_{N \to \infty} A_N \geq 1$, and $\lim_{N \to \infty} c_N \leq c$. Moreover, from Proposition \ref{IKprop} and from the minimality of $I_K(z)$, one has $c_N^2 = I_K(z_N') \geq I_K(z) = c^2$, and hence $\lim_{N \to \infty} c_N = c$ and consequently $\lim_{N \to \infty} I_K(z_N') = I_K(z)$, and $\lim_{N \to \infty} A_N = 1$. (Note that $z_N'$ satisfies the conditions of Proposition \ref{IKprop} because each single $p_i$ trivially satisfies that the face $F_i$ of $K$ is non-empty.)

Consider the space $\mathcal{E}^1$ of piecewise affine curves $z'$, whose velocities are in the set $C \cdot \{p_1,\ldots,p_{_{\kF}}\}$ for some $C>0$ and each $p_i$ appears at most once. 
Let us define a map $\Phi : \mathcal{E}^1 \to S_{_{\kF}} \times \R^{\kF}$, $z' \mapsto (\sigma,(|I_1|,\ldots,|I_{_{\kF}}|))$, where
\[ \dz'(t) = \sum_{i=1}^{\kF} \Id_{I_i}(t) C \cdot p_{\sigma(i)} .\] 
A point in the image $(\sigma,(t_1,\ldots,t_{_{\kF}})) \in \text{Im}(\Phi)$ satisfies $t_i \geq 0$ for each $i$, and $\sum_{i=1}^{\kF} t_i = 1$, which implies that $\text{Im}(\Phi)$ belongs to a compact set in the usual topology.
Note that $z'_N \in \mathcal{E}^1$ with $C = c_N$. Suppose that $\Phi(z'_N) = (\sigma^N,(t^N_1,\ldots,t^N_{_{\kF}}))$, then after passing to a subsequence, one can assume that $\sigma^N = \sigma$ is constant, and $(t^N_1,\ldots,t^N_{_{\kF}})$ converges to a vector $(t^\infty_1,\ldots,t^\infty_{_{\kF}})$.
Let $z'_\infty$ be the piecewise affine curve identified with $(\sigma,(t^\infty_1,\ldots,t^\infty_{\kF}))$, and with $C = \lim_{N \to \infty} c_N = c$. 
Note that $\|z'_N - z'_\infty \|_{1,2} \to 0$.
Indeed, let $\mathcal{T}^N \subset [0,1]$ be the set of times where $\dz'_N(t) = \frac{c}{c_N} \dz'_\infty(t)$. 
Since $c_N \xrightarrow{N \to \infty} c$, one has 
$\int_{\mathcal{T}^N} \| \dz'_N(t) - \dz'_\infty(t)\|^2 dt \xrightarrow{N \to \infty} 0$. 
Note that for each $t \in [0,1]$ such that $\dz'_N(t)$ and $\dz'_\infty(t)$ are defined, $\|\dz'_N(t) - \dz'_\infty(t)\|^2$ is bounded, since both belong to a finite set of velocities and $c_N$ is bounded. 
Hence since $| \mathcal{T}^N| \xrightarrow{N \to \infty} 1$, 
one has $\int_{[0,1] \setminus \mathcal{T}^N} \|\dz'_N(t) - \dz'_\infty(t)\|^2 dt \xrightarrow{N \to \infty} 0$.
Moreover, since $\int_0^1 \dz'_N(t) dt = 0$ for each $N$, one gets that $\int_0^1 \dz'_\infty(t) dt = 0$ and hence $z'_\infty$ is a closed loop. Similarly, one can check that $z'_\infty \in \mathcal{E}$, and finally by Proposition \ref{IKprop}, $I_K(z'_\infty) = c^2 = I_K(z)$. Since $z$ was chosen to be a minimizer, we get that $z'_\infty$ is also a minimizer, and therefore it is a weak critical point of $I_K$, i.e. $z'_\infty \in \mathcal{E}^\dagger$. Finally by invoking Lemma \ref{dual_bijection_lemma}, one gets a piecewise affine closed characteristic $\gamma'$ where $\dg'(t) \in d \cdot \{p_1,\ldots,p_{_{\kF}}\}$ outside a finite subset of $[0,1]$, and the set $\{ t : \dg'(t) = d p_i \}$ is connected for every $i$, i.e. every velocity $p_i$ appears at most once.

\end{proof}

We are now in a position to prove Theorem \ref{formula_theorem}.

\begin{proof}[Proof of Theorem \ref{formula_theorem}]
Let $K$ be a convex polytope. From Lemma \ref{dual_bijection_lemma} it follows that 
\begin{equation}
\label{dual_capacity_formula}
\ehzcap(K) = \myMin{z \in \mathcal{E}} 2 I_K(z).
\end{equation}
Theorem \ref{simple_loop_theorem} implies that there exists $z \in \mathcal{E}$ which minimizes $I_K$ and is of the form
\[ \dz(t) = \sum_{i=1}^{\kF} \Id_{I_i} cp_{\sigma(i)} .\]
for some $\sigma \in S_{\kF}$, and $c > 0$. Therefore, when calculating the minimum in \eqref{dual_capacity_formula}, one can restrict to loops of this form in $\mathcal{E}$. Let us rewrite the conditions for $z$ to be in $\mathcal{E}$ in this case.
The condition $\int_0^1 \dz(t) dt = 0$ is equivalent to $\sum_{i=1}^{\kF} T_i p_{\sigma(i)} = 0$, where we denote $T_i = |I_i|$. 
By means of Proposition \ref{technicalProp} the condition $\int_0^1 \langle -J \dz(t), z(t) \rangle dt = 1$ can be written as
\[ 1 = \int_0^1 \langle -J\dz(t),z(t) \rangle dt = c^2 \sum_{1 \leq j < i \leq \kF} T_i T_j \omega( p_{\sigma(i)}, p_{\sigma(j)}) .\]
Finally by Proposition \ref{IKprop},
\[ I_K(z) = c^2. \]
Overall we get that
\[ \ehzcap(K) = 2\myMin{ \substack{  (T_i) \in M^T(K) \text{ s.t.} \\ \forall \sigma \in S_k, c^2 A_K(\sigma,(T_i)) \leq 1  } } c^2, \]
where
\[ M^T(K) = \left\{ (T_i)_{i=1}^{\kF} : T_i \geq 0, \sum_{i=1}^{\kF} T_i = 1, \sum_{i=1}^{\kF} T_i p_{\sigma(i)} = 0 \right\}, \]
and 
\[ A_K(\sigma,(T_i)_{i=1}^{\kF}) = \sum_{1 \leq j < i \leq \kF} T_i T_j \omega( p_{\sigma(i)}, p_{\sigma(j)}). \]
This can be written as
\[ \ehzcap(K) = 2 \left[ \max_{ \substack{ \sigma \in S_{\kF} \\ (T_i)_{i=1}^{\kF} \in M^T(K) } } \sum_{1 \leq j < i \leq \kF} T_i T_j \omega( p_{\sigma(i)}, p_{\sigma(j)})\right]^{-1} .\]
Since $p_i = \frac{2}{h_i}Jn_i$, we can set $\beta_{\sigma(i)} = \frac{T_i}{h_{\sigma(i)}}$ and get the required formula.
\end{proof}

\begin{remark}
By plugging the simple closed characteristic from Theorem \ref{simple_loop_theorem} in the formula for $\ehzcap$ from Remark \ref{abbondandolo_remark}, one gets a similar proof for Theorem \ref{formula_theorem}.
\end{remark}

\begin{remark}
\label{normals_repititions_remark}
From the proof of Theorem \ref{formula_theorem} we see that if one considers loops $z \in \mathcal{E}$ with $\dz$ piecewise constant, and whose velocities are of the form $d p_i$, without the restriction that each $p_i$ appears at most once, one still gets an upper bound for $\ehzcap(K)$.
More precisely each selection of a sequence of unit outer normals to facets of $K$ $(n_i)_{i=1}^m$ and a sequence of numbers $(\beta_i)_{i=1}^m$ that satisfy 
\[ \beta_i \geq 0, \quad \sum_{i=1}^m \beta_i h_K(n_i) = 1, \quad \sum_{i=1}^m \beta_i n_i = 0, \] gives an upper bound of the form
\[ \ehzcap(K) \leq \frac{1}{2} \left[ \sum_{1 \leq j < i \leq m} \beta_i \beta_j \omega(n_i,n_j) \right]^{-1} .\]
This fact will be useful for us in the proof of Theorem \ref{subadditivity_theorem}.
\end{remark}

\begin{remark}
\label{translation_invariant_remark}
Note that formula \eqref{formula_equation} for $\ehzcap$ in Theorem \ref{formula_theorem} is invariant under translations and is $2$-homogeneous. Indeed, if we take $\widetilde{K} = K + x_0$ we get the same normals and the oriented heights change to $\widetilde{h}_i = h_i + \langle x_0,n_i \rangle$. For $( \beta_i )_{i=1}^{\kF} \in M(K)$, one can check that $\sum \beta_i \widetilde{h}_i = \sum \beta_i h_i + \langle x_0, \sum \beta_i n_i \rangle = 1$. Hence $( \beta_i )_{i=1}^{\kF} \in M(\widetilde{K})$ so we get the same value for 
\[ \mySum{1 \leq j < i \leq \kF}{} \beta_{\sigma(i)} \beta_{\sigma(j)} \omega(n_{\sigma(i)},n_{\sigma(j)}).  \]
\\Hence $\ehzcap(K) = \ehzcap(\widetilde{K})$.

On the other hand, consider $\widetilde{K} = \lambda K$ for some $\lambda > 0$, then it has the same normals as $K$, and the oriented heights change to $\widetilde{h}_i = \lambda h_i$. For $(\beta_i)_{i=1}^{\kF} \in M(K)$, take $\widetilde{\beta}_i = \frac{\beta_i}{\lambda}$, to get $(\widetilde{\beta}_i)_{i=1}^{\kF} \in M(\widetilde{K})$. We get that 
\[ \mySum{1 \leq j < i \leq \kF}{} \widetilde{\beta}_{\sigma(i)} \widetilde{\beta}_{\sigma(j)} \omega(n_{\sigma(i)},n_{\sigma(j)}) = \frac{1}{\lambda^2}\mySum{1 \leq j < i \leq \kF}{} \beta_{\sigma(i)} \beta_{\sigma(j)} \omega(n_{\sigma(i)},n_{\sigma(j)}) .\]
\\Hence, $\ehzcap(\widetilde{K}) = \lambda^2 \ehzcap(K)$.
\end{remark}

\begin{remark}		
Formula $\eqref{formula_equation}$ is invariant under multiplication by a symplectic matrix $A \in \text{Sp}(2n)$. Indeed, take $\widetilde{K} = AK$. The new normals are $\widetilde{n}_i = \frac{(A^t)^{-1} n_i} {\| (A^t)^{-1} n_i\|}$, and the new oriented heights are $\widetilde{h}_i = \frac{h_i}{\|(A^t)^{-1} n_i\|}$. 		
One can take $\widetilde{\beta}_i = \frac{\beta_i}{\|(A^t)^{-1} n_i\|}$ and get that $\ehzcap(K) = \ehzcap(\widetilde{K})$.		
\end{remark}

\begin{remark}
The number of permutations in $S_{\kF}$ grows exponentially in $\kF$ and thus can be a huge number. For computational goals, it is worth noting that this set can be reduced. Consider a directed graph $G$, with vertex set $\{j\}$ corresponding to facets of $K$, $\{F_j\}$, and where there exists an edge $ij$ if there exists a point $x \in F_i$, and a constant $c > 0$ so that $x + c p_i \in F_j$.
Denote by $A$ the set of all cycles on $G$. An element $I \in A$ is a sequence $(I(1),\ldots,I(l))$, where there are edges $I(i)I(i+1)$ for $i=1,\ldots,l-1$ and there is an edge $I(l)I(1)$. We get that
\[\ehzcap(K) = \frac{1}{2} \left[ \max_{I \in A, (\beta_i) \in M_I(K)} \mySum{1 \leq j < i \leq |I|}{} \beta_{I(i)} \beta_{I(j)} \omega(n_{I(i)},n_{I(j)}) \right]^{-1} ,\]
where
\[ M_I(K) = \left\{ (\beta_i)_{i=1}^{|I|} : \beta_i \geq 0, \sum_{i=1}^{|I|} \beta_{I(i)} h_{I(i)} = 1, \sum_{i=1}^{|I|} \beta_{I(i)} n_{I(i)} = 0 \right\}. \]

\end{remark}

\begin{proof}[Proof of Corollary \ref{centrally_symmetric_formula_theorem}]
Let $K \subset \R^{2n}$ be a convex polytope that satisfies $K = -K$. Let $n_1,\ldots,n_{_{\kF'}},-n_1,\ldots,-n_{_{\kF'}}$ be the normals to the $(2n-1)$-dimensional facets of $K$.
Recall that $p_i = J \partial g_K^2|_{F_i} = \frac{2}{h_i} J n_i$. 
By Theorem \ref{simple_loop_theorem}, there exists a closed characteristic $\gamma$ on the boundary of $K$ whose velocities are piecewise constant, and are a positive multiple of $\{\pm p_i \}_{i=1}^{\kF'}$, so that for each $i$, the velocity which is a positive multiple of $p_i$ (and the one which is a positive multiple of $-p_i$) appears at most once.
Consider a reparametrization of $\gamma$ such that $\dg(t) \in d \{ \pm p_i \}$ almost everywhere, for some $d > 0$ independent of $i$.
From Lemma \ref{dual_bijection_lemma} there exists a corresponding element $z \in \mathcal{E}^\dagger$, such that $z = \lambda \gamma + b$ and $z$ is a minimizer of $I_K$. 
The velocities of $z$ are positive multiples of the velocities of $\gamma$ and hence have the same properties. The idea of the proof is to change $z$ to $z'$ so that $z'$ would also be a minimizer whose velocities have the same properties, and which satisfies $z'(t+\frac{1}{2}) = -z'(t)$.
The next argument (see \cite{Karasev_private}) was communicated to us by R. Karasev, we include it here for completeness.

Translate $z$ so that $z(0) = -z(\frac{1}{2})$. Since $\int_0^1 \langle -J \dz(t),z(t) \rangle dt = 1$, we either have $\int_0^{\frac{1}{2}} \langle -J \dz(t),z(t) \rangle dt \geq \frac{1}{2}$, or $\int_{\frac{1}{2}}^1 \langle -J \dz(t),z(t) \rangle dt \geq \frac{1}{2}$.
Assume without loss of generality that the first inequality holds, i.e.
\[ \int_0^{\frac{1}{2}} \langle -J \dz(t),z(t) \rangle dt \geq \frac{1}{2}. \]
Define
\[ z' = \left\{ \begin{matrix} z(t), & & t \in [0,\frac{1}{2}] \\
-z(t - \frac{1}{2}), & & t \in [\frac{1}{2},1] \end{matrix} \right. \]
Since $\dz(t) \in c \{ \pm p_i \}_{i=1}^{\kF'} = c \{ \pm \frac{2}{h_i} Jn_i \}_{i=1}^{\kF'} $, where $c = \lambda d$, one has $h_K^2 (-J \dz(t)) = 4c^2$ . 
Note that since $K = -K$ one has $h_K(x) = h_K(-x)$ for all $x \in \R^{2n}$, hence one gets 
\[ I_K(z) = \frac{1}{4} \int_0^1 h_K^2 (-J \dz(t)) dt = c^2 = \frac{1}{2} \int_0^{\frac{1}{2}} h_K^2 (-J \dz(t)) dt = I_K(z') .\]
Moreover, 
\[ \int_0^1 \dz'(t) dt = \int_0^{\frac{1}{2}} \dz(t) dt - \int_0^{\frac{1}{2}} \dz(t) dt = 0,\]
and
\[ \int_0^1 \langle -J \dz'(t) , z'(t) \rangle dt = 2 \int_0^{\frac{1}{2}} \langle -J \dz(t),z(t) \rangle dt \geq 1 .\]
Hence one can divide $z'$ by a constant to get $\int_0^1 \langle -J \dz'(t) , z'(t) \rangle dt = 1$ and $I_K(z') \leq I_K(z)$. Since $z$ was chosen to be a minimizer, the constant must be $1$, and $I_K(z') = I_K(z)$. Hence $z'$ is a minimizer that satisfies $z' = -z'$.

After plugging $z'$ in Formula \eqref{formula_equation} for $\ehzcap(K)$ from Theorem \ref{formula_theorem} we get a maximum, hence there exists an order of the normals that gives maximum in \eqref{formula_equation} which has the following form.
\[ a(1) n_{\sigma(1)},\ldots,a(\kF') n_{\sigma(\kF')},-a(1) n_{\sigma(1)},\ldots,-a(\kF') n_{\sigma(\kF')} ,\]
where $a(i) = \pm 1$, and $\sigma \in S_{\kF'}$. Recall that here the number of facets is $2 \kF'$. In addition, since $\beta_i = \frac{T_i}{h_i}$ (see the proof of Theorem \ref{formula_theorem}), from the symmetry of $z'$ the oriented heights $h_i$ and the times $T_i$ in the first half are equal to the oriented heights and the times in the second half, and hence the ``betas" in the first half are equal to the ``betas" in the second half. 
Let us consider the sum we try to maximize in \eqref{formula_equation} :
\begin{eqnarray*}
&\sum_{1 \leq i < j \leq \kF'} & \beta_{\sigma(i)} \beta_{\sigma(j)} ( \omega(a(i)n_{\sigma(i)},a(j)n_{\sigma(j)})
+ \omega(-a(i)n_{\sigma(i)},-a(j)n_{\sigma(j)}) ) \\
&&+ \sum_{i=1}^{\kF'} \beta_{\sigma(i)} a(i) \sum_{j=1}^{\kF'} \beta_{\sigma(j)} a(j) \omega(n_{\sigma(i)},-n_{\sigma(j)}) \\
&= 2 \sum_{1 \leq i < j \leq \kF'} & \beta_{\sigma(i)} \beta_{\sigma(j)}\omega(a(i)n_{\sigma(i)},a(j)n_{\sigma(j)}) \\ 
&&+ \omega( \sum_{i=1}^{\kF'} \beta_{\sigma(i)} a(i) n_{\sigma(i)} , -\sum_{i=1}^{\kF'} \beta_{\sigma(i)} a(i) n_{\sigma(i)})\\
&= 2 \sum_{1 \leq i < j \leq \kF'} & \beta_{\sigma(i)} \beta_{\sigma(j)}\omega(a(i)n_{\sigma(i)},a(j)n_{\sigma(j)}).
\end{eqnarray*} 
We get that the sum we try to maximize in \eqref{formula_equation} is equal to twice the sum over the normals in the first half. In addition, in $M(K)$ we can remove the constraint $\sum_{i=1}^{2\kF'} \beta_i n_i = 0$ because we get it automatically (since the second half of the normals are minus the first half and the ``betas" are equal). The constraint $\sum_{i=1}^{2\kF'} \beta_i h_i = 1$ becomes $\sum_{i=1}^{\kF'} \beta_i h_i = \frac{1}{2}$ and instead of considering the constraints $\beta_i \geq 0$ for each $i$, we can remove the signs $a(i)$ from the normals, and allow for negative ``betas" as well. In conclusion, we get that the only constraint for the ``betas" is $\sum_{i=1}^{\kF'} |\beta_i| h_i = \frac{1}{2}$ and this gives us the formula we need and thus proves Corollary \ref{centrally_symmetric_formula_theorem}.
\end{proof}
\section{Subadditivity for hyperplane cuts}
\label{more_proofs_section}

In the proof of Theorem \ref{subadditivity_theorem}, we use the formula for the capacity that was proved in Theorem \ref{formula_theorem}, in its equivalent formulation which was given in Remark \ref{reformulationRemark}, namely,
\begin{equation}
\label{formula2_equation}
\ehzcap(K) = \frac{1}{2} \min_{ (\beta_i,n_i)_{i=1}^{\kF} \in M_2(K)  } \left( \sum_{i=1}^{\kF} \beta_i h_K(n_i) \right)^2, 
\end{equation} 
where
\[ M_2(K) = 
\left\{
	\begin{array}{ll}
		(\beta_i,n_i)_{i=1}^{\kF} : \beta_i \geq 0, (n_i)_{i=1}^{\kF} \text{ are different outer normals to }K\\ \sum_{i=1}^{\kF} \beta_i n_i = 0, \quad\quad
		\sum_{1 \leq j < i \leq \kF} \beta_i \beta_j \omega(n_i,n_j) = 1 \\
	\end{array}
\right\} .\]

To see that this is indeed equivalent to the form given in Theorem \ref{formula_theorem}, note that
\begin{align*}
\ehzcap(K) &= \frac{1}{2} \left[ \max_{\sigma \in S_{\kF}, (\beta_i) \in M(K)}  \mySum{1 \leq j < i \leq \kF}{} \beta_{\sigma(i)} \beta_{\sigma(j)} \omega(n_{\sigma(i)},n_{\sigma(j)}) \right]^{-1} \\
&= \frac{1}{2} \left[ \max_{\sigma \in S_{\kF}, \beta_i \geq 0, \sum_{i=1}^{\kF} \beta_i n_i = 0} \frac{\sum_{1 \leq j < i \leq \kF} \beta_{\sigma(i)} \beta_{\sigma(j)} \omega(n_{\sigma(i)},n_{\sigma(j)}) }{  \left( \sum_{i=1}^{\kF} \beta_i h_i \right)^2 } \right]^{-1} \\
&= \frac{1}{2} \left[ \max_{(\beta_i, n_i) \in M_2(K)} \frac{1}{\left( \sum_{i=1}^{\kF} \beta_i h_K(n_i) \right)^2} \right]^{-1} \\
&= \frac{1}{2} \min_{(\beta_i,n_i) \in M_2(K)} \left( \sum_{i=1}^{\kF} \beta_i h_K(n_i) \right)^2.
\end{align*} 

Before providing the full proof of Theorem \ref{subadditivity_theorem}, let us briefly describe the main idea.
Suppose we cut a convex polytope $K$ by a hyperplane $H$ into $K_1$ and $K_2$. 
Our strategy is to take minimizers in $M_2(K_1)$ and in $M_2(K_2)$, and construct from them a sequence of normals and coefficients on $K$ that gives an upper bound for $\ehzcap(K)$ which is less than or equal to $\ehzcap(K_1)+\ehzcap(K_2)$. 
By Theorem \ref{simple_loop_theorem}, we know that one can take the minimizers so that the normal to the shared facet $K_1 \cap H = K_2 \cap H$ appears at most once. This enables us to choose coefficients so that this normal in both minimizers cancels out and we are left with a minimizer in $M_2(K)$.

\begin{proof}[Proof of Theorem \ref{subadditivity_theorem}]
From the continuity of the EHZ capacity (see e.g. \cite{mcDuff_Salamon}, Exercise 12.7) it is enough to prove the statement for polytopes. Let $K \subset \R^{2n}$ be a convex polytope.

Suppose we cut $K$ by a hyperplane into $K_1$ and $K_2$.
Without loss of generality, choose the origin to be on the hyperplane that divides $K$ into $K_1$ and $K_2$. 
Choose $( \beta_i,n_i )_{i=1}^{\bF{K_1}}, ( \alpha_i,w_i )_{i=1}^{\bF{K_2}}$ to be minimizers in Equation \eqref{formula2_equation} for $\ehzcap(K_1)$ and $\ehzcap(K_2)$ respectively. 
In addition, denote by $n$ the normal to the hyperplane splitting $K$ into $K_1$ and $K_2$ where we choose the positive direction to go into $K_1$.
Note that for each outer normal $n_i \neq n$ of $K_1$, one has $h_{K_1}(n_i) = h_K(n_i)$, and for each outer normal $w_i \neq n$ of $K_2$, one has $h_{K_2}(w_i) = h_K(w_i)$. In addition, one has $h_{K_1}(n) = h_{K_2}(n) = 0$.
Assume without loss of generality that $n_1 = -n$ and $w_{_{\bF{K_2}}} = n$ (this can be assumed because one can always take cyclic permutations of the sequences to get new sequences that satisfy the constraints and give the same result). 
By means of Theorem \ref{simple_loop_theorem}, each normal vector appears at most once, and hence for each $i \neq 1$, $n_i \neq n$, and for each $i \neq \bF{K_2}$, $w_i \neq n$. 
First note that if $\beta_1 = 0$ or $\alpha_{_{\bF{K_2}}} = 0$ we are done. 
Indeed, suppose that $\beta_1 = 0$. 
All the normals $n_i$ for $i \geq 2$ are normals also to facets of $K$. 
Hence $\beta_1=0$ implies $(\beta_i,n_i)_{i=2}^{\bF{K_1}} \in M_2(K)$ (after adding the rest of the normals with coefficients zero), and this gives $\ehzcap(K) \leq \ehzcap(K_1)$. 
From now on  assume $\beta_1 \neq 0$ and similarly $\alpha_{_{\bF{K_2}}} \neq 0$. 
Next, consider the following sequence of coefficients
\begin{align*}
(\delta_i)_{i=1}^{\bF{K_2}+\bF{K_1}-2} := 
\left( \frac{\beta_1}{\sqba} \alpha_1,\ldots,\frac{\beta_1}{\sqba} \alpha_{{_{\bF{K_2}-1}}},\frac{\alpha_{_{\bF{K_2}}}}{\sqba} \beta_2,\ldots,\frac{\alpha_{_{\bF{K_2}}}}{\sqba} \beta_{_{\bF{K_1}}} \right),
\end{align*}
and the sequence of normals $( u_i )_{i=1}^{\bF{K_1}+\bF{K_2}-2} := (w_1,\ldots,w_{_{\bF{K_2}-1}},n_2,\ldots,n_{_{\bF{K_1}}})$, where $\sqba := \sqrt{\beta_1^2 + \alpha_{_{\bF{K_2}}}^2}$. Note that here we allow for repetitions of the normals and we may have $\bF{K_1}+\bF{K_2}-2 > \kF$. 
However, from Remark \ref{normals_repititions_remark} we know that if one considers any sequence $(\delta_i,u_i)_{i=1}^{m}$ with $\delta_i \geq 0$, $u_i$ a normal to $K$, that satisfies the constraints $\sum_{i=1}^{m} \delta_i u_i = 0$ and $\sum_{1 \leq j < i \leq m} \delta_i \delta_j \omega(u_i,u_j) = 1$ for any $m \in \mathbb{N}$, the value
\[ \frac{1}{2} \left( \sum_{i=1}^{m} \delta_i h_K(u_i) \right)^2\]
still gives an upper bound for $\ehzcap(K)$.
Hence we wish to show that $(\delta_i,u_i)_{i=1}^{\bF{K_1}+\bF{K_2}-2}$ satisfies the constraints for $K$.
First note that since $\sum_{i=1}^{\bF{K_2}} \alpha_i w_i = 0$, $\sum_{i=1}^{\bF{K_1}} \beta_i n_i = 0$ one has
\[ \sum_{i=1}^{\bF{K_1}+\bF{K_2}-2} \delta_i u_i = - \frac{\beta_1}{\sqba} \cdot \alpha_{_{\bF{K_2}}} n +  \frac{\alpha_{_{\bF{K_2}}}}{\sqba} \cdot \beta_1 n = 0 .\]
 Next, note that
\[ \sum_{1 \leq j < i \leq \bF{K_1}+\bF{K_2}-2} \delta_i \delta_j \omega(u_i,u_j) =  \frac{\beta_1^2}{\sqba^2} \sum_{1 \leq j < i \leq \bF{K_2}-1} \alpha_i \alpha_j \omega(w_i,w_j) + \]
\[ \frac{\alpha_{_{\bF{K_2}}}^2}{\sqba^2} \sum_{2 \leq j < i \leq \bF{K_1}} \beta_i \beta_j \omega(n_i,n_j) + \frac{\beta_1\alpha_{_{\bF{K_2}}}}{\sqba^2} \sum_{i=2}^{\bF{K_1}} \sum_{j=1}^{\bF{K_2}-1} \beta_i \alpha_j \omega(n_i,w_j). \] 
Since $ -\alpha_{_{\bF{K_2}}} w_{_{\bF{K_2}}} = \sum_{i=1}^{\bF{K_2}-1} \alpha_i w_i $, one has $\omega( \alpha_{_{\bF{K_2}}} w_{_{\bF{K_2}}}, \sum_{i=1}^{\bF{K_2}-1} \alpha_i w_i) = 0$, and therefore
\[ \sum_{1 \leq j < i \leq \bF{K_2}-1} \alpha_i \alpha_j \omega(w_i,w_j) = \sum_{1 \leq j < i \leq \bF{K_2}} \alpha_i \alpha_j \omega(w_i,w_j) = 1. \]
Similarly,
\[ \sum_{2 \leq j < i \leq \bF{K_1}} \beta_i \beta_j \omega(n_i,n_j)  = 1. \]
Finally, note that
\[ \sum_{i=2}^{\bF{K_1}} \sum_{j=1}^{\bF{K_2}-1} \beta_i \alpha_j \omega(n_i,w_j) = \omega( \sum_{i=2}^{\bF{K_1}} \beta_i n_i, \sum_{j=1}^{\bF{K_2}-1} \alpha_i w_i) = -\beta_1 \alpha_{_{\bF{K_2}}} \omega(n,n) = 0.\]
Overall we get that 
\[ \sum_{1 \leq j < i \leq \bF{K_1}+\bF{K_2}-2} \delta_i \delta_j \omega(u_i,u_j) =  \frac{\beta_1^2}{\sqba^2} + \frac{\alpha_{_{\bF{K_2}}}^2}{\sqba^2} = 1. \] 
Hence $( \delta_i , u_i )_{i=1}^{\bF{K_1}+\bF{K_2}-2}$ indeed satisfies the constraints, and therefore $\ehzcap(K)$ is less than or equal to $\frac{1}{2} \left(\sum_{i=1}^{\bF{K_1}+\bF{K_2}-2} \delta_i h_K(u_i) \right)^2$. It remains to show that 
\[ \frac{1}{2} \left(\sum_{i=1}^{\bF{K_1}+\bF{K_2}-2} \delta_i h_K(u_i) \right)^2 \leq \frac{1}{2} \left(\sum_{i=1}^{\bF{K_2}} \alpha_i h_{K_2}(w_i) \right)^2 + \frac{1}{2} \left(\sum_{i=1}^{\bF{K_1}} \beta_i h_{K_1}(n_i) \right)^2. \]
A straightforward computation gives
\begin{eqnarray*}
 \left(\sum_{i=1}^{\bF{K_1}+\bF{K_2}-2} \delta_i h_K(u_i) \right)^2 &=& \frac{1}{\sqba^2} \left( \beta_1^2 \left(\sum_{i=1}^{\bF{K_2}-1} \alpha_i h_{K_2}(w_i) \right)^2 \right. \\
 &+& 2 \beta_1 \alpha_{_{\bF{K_2}}} \sum_{i=1}^{\bF{K_2}-1} \alpha_i h_{K_2}(w_i) \sum_{i=2}^{\bF{K_1}} \beta_i h_{K_1}(n_i) \\
 &+& \left. \alpha_{_{\bF{K_2}}}^2 \left(\sum_{i=2}^{\bF{K_1}} \beta_i h_{K_1}(n_i) \right)^2 \right).
\end{eqnarray*}
Since
\begin{align*} 
2 \beta_1 \alpha_{_{\bF{K_2}}} \sum_{i=1}^{\bF{K_2}-1} \alpha_i h_{K_2}(w_i) & \sum_{i=2}^{\bF{K_1}} \beta_i h_{K_1}(n_i)  \\
& \leq \alpha_{_{\bF{K_2}}}^2 \left( \sum_{i=1}^{\bF{K_2}-1} \alpha_i h_{K_2}(w_i) \right)^2 + \beta_1^2 \left( \sum_{i=2}^{\bF{K_1}} \beta_i h_{K_1}(n_i) \right)^2, 
\end{align*}
one has
\begin{align*}
\ehzcap(K) &\leq \frac{1}{2} \left(\sum_{i=1}^{\bF{K_1}+\bF{K_2}-2} \delta_i h_K(u_i) \right)^2  \\
&\leq \frac{1}{2\sqba^2} \left(  (\beta_1^2 + \alpha_{_{\bF{K_2}}}^2) \left(\sum_{i=1}^{\bF{K_2}-1} \alpha_i h_{K_2}(w_i) \right)^2 + (\beta_1^2 + \alpha_{_{\bF{K_2}}}^2) \left(\sum_{i=2}^{\bF{K_1}} \beta_i h_{K_1}(n_i) \right)^2 \right) \\
&= \frac{1}{2} \left(\sum_{i=1}^{\bF{K_2}-1} \alpha_i h_{K_2}(w_i)  \right)^2 + \frac{1}{2}\left(\sum_{i=2}^{\bF{K_1}} \beta_i h_{K_1}(n_i) \right)^2  \\
&= \frac{1}{2}\left(\sum_{i=1}^{\bF{K_2}} \alpha_i h_{K_2}(w_i) \right)^2 + \frac{1}{2} \left(\sum_{i=1}^{\bF{K_1}} \beta_i h_{K_1}(n_i) \right)^2 
= \ehzcap(K_1) + \ehzcap(K_2), \\
\end{align*}
where the second to last equality is due to the fact that 
\[h_{K_1}(n) = h_{K_2}(n) = 0. \]
\end{proof}

\bibliography{references}
\bibliographystyle{siam}

\vspace{1cm}
\noindent
School of Mathematical Sciences\\
Tel Aviv University\\
Tel Aviv 6997801, Israel\\
pazithaim@mail.tau.ac.il

\end{document}